\documentclass{amsart}
\usepackage{graphicx}
\usepackage{amsfonts}
\newtheorem{tm}{Theorem}

\newtheorem{defi}{Definition}

\newtheorem{rem}{Remark}
\newtheorem{rems}{Remarks}
\newtheorem{lm}{Lemma}
\newtheorem{ex}{Example}

\newtheorem{prop}{Proposition}

\newtheorem{problem}{Problem}

\newcommand{\sgn}{{\rm sgn}}
\newcommand{\bsi}{\overline \sigma}

\newcommand{\eps}{\epsilon}
\newcommand{\si}{\sigma}
\newcommand{\tsi}{\widehat \si}

\begin{document}

\title{New aspects of Descartes' rule of signs}
\author{Vladimir Petrov Kostov, Boris Zalmanovich Shapiro}
 \address{Universit\'e C\^ote d'Azur, LJAD, France} 
\email{vladimir.kostov@unice.fr}
\address{Stockholm University, S-10691, Stockholm, Sweden}
\email{shapiro@math.su.se}

\begin{abstract}
Below we summarize some new developments in the area of distribution of roots 
and signs of real univariate polynomials  pioneered by R.~Descartes in the 
middle of the 17-th century. 
\end{abstract}

\dedicatory{To Ren\'e Descartes,  a polymath in philosophy and science }
\date{}


\maketitle 

\section{Introduction}\label{sec:intro}

The classical Descartes' rule of signs claims that 
the number of positive roots of a real univariate polynomial is 
bounded by the number of sign changes in the sequence of its coefficients and it coincides with the latter number modulo $2$. 
It was published in French (instead of the usual at that time Latin) as a small portion of ``Sur la construction de probl\`emes solides ou plus que solide"  which is the third book of Descartes'  fundamental treatise ``La G\'eom\'etrie"   which, in its turn, is an appendix to his famous ``Discours de la m\'ethode".  It is in the latter chef d'oeuvre that Descartes developed his  analytic approach to geometric problems leaving practically all proofs  and details to an interested reader.  This interested reader turned out to be Frans van Schooten, a professor of mathematics at Leiden who together with his students undertook a tedious work of making Descartes' writings understandable, translating and  publishing them in the proper language,  that is Latin. (For the electronic version of this book see \cite {Gut}.)   Mathematical achievements of Descartes form a small 
fraction of his overall scientific and philosophical legacy and Descartes' rule of signs is a small but 
important fraction of his mathematical  heritage. %R.~Descartes' life with its ups and downs is worth to be better known.  

Descartes' rule of signs has been studied and generalized by many authors over the years;  one of the earliest can be found in \cite{Fo}, see also \cite{Ca} and \cite{Ga}.  
(For some recent contributions, see \cite{AlFu, AJS,  Di,   FNoSh,  Gr, Ha}, to mention a few.) 

\medskip
In the present survey we summarize a relatively new development in this 
area which, to the best of our knowledge, was initiated  only in the 1990s, 
see \cite{Gr}. %(But  with such elementary questions one can never be completely sure.) 

\medskip
For simplicity, we consider below only real univariate polynomials  with 
all non-vanishing coefficients. 
For a polynomial $P:=\sum _{j=0}^da_jx^j$ with fixed signs of its coefficients,  Descartes' rule of signs tells 
us what  possible values the number  of its real positive roots can have. For $P$ as above,  
we 
define the sequence of $\pm$ signs of length $d+1$ which we call 
the {\em sign pattern} (SP for short) of $P$. Namely, we say that a polynomial $P$ with all non-vanishing coefficients {\em defines 
the sign pattern $\sigma :=(s_d$, $s_{d-1}$, $\ldots$, $s_0)$} if 
$s_j=\sgn\; a_j$. 
Since the roots of the polynomials $P$  and $-P$ are the same, 
we can without loss of generality, assume that the first sign 
of a SP is always a~$+$.

It is true that for a given SP with $c$ sign changes 
(and hence with $p=d-c$ sign preservations), 
there always exist polynomials $P$ defining this sign pattern and having exactly 
$pos$ positive roots, where $pos=0, 2, \ldots, c$ if $c$ is even and 
$pos=1, 3, \ldots, c$ if $c$ is odd, see e.g. \cite{AlFu, Av}. (Observe that we do not 
impose 
 any restriction on the number of negative roots of these polynomials.) 

One can apply Descartes' rule of signs to the polynomial 
$(-1)^dP(-x)$ 
which has $p$ sign changes and $c$ sign preservations in the sequence of 
its coefficients and whose leading coefficient is positive. 
The roots of $(-1)^dP(-x)$ are obtained from the roots of $P(x)$ by changing their sign. 
Applying the above result of \cite{AlFu} to $(-1)^dP(-x)$ 
one obtains the existence 
of polynomials $P$ with exactly $neg$ negative roots, where 
$neg=0, 2, \ldots, p$ if $p$ is even and 
$neg=1, 3, \ldots, p$ if $p$ is odd.  (Here again we impose no requirement 
on the number of positive roots).      

\medskip
A natural question apparently for the first time raised in \cite{Gr} is whether one can  freely combine 
these two results about 
the numbers of positive and  negative roots. Namely, given a SP $\sigma$ with $c$ sign 
changes and $p=d-c$ sign preservations, we define its {\em admissible   pair} (AP for short) 
  as  $(pos, neg)$, where 
$pos \leq c$, $neg \leq p$ and the differences $c-pos$ and $p-neg$ are even. 
For the SP $\sigma$ as above, we call   $(c, p)$ the {\em Descartes' pair} of $\sigma$. The main question under consideration in this paper is as follows. 

\begin{problem}\label{pb1}
 Given a couple (SP, AP), does there exist a  
polynomial of degree $d$ with this SP and having exactly $pos$ positive 
and exactly $neg$ 
negative roots (and hence exactly $(d-pos-neg)/2$ complex conjugate pairs)?
\end{problem}

If such a polynomial exists, then we say that it {\em realizes} a given 
couple (SP, AP). 
The present paper discusses the current status of knowledge in this 
realization problem. 

\begin{ex}{\rm For $d=4$ and for the sign pattern $\sigma ^0:=(+,-,-,-,+)$,  
the following pairs and only them are admissible: 
$(2,2)$, $(2,0)$, $(0,2)$ and $(0,0)$.  The 
first of them is the Descartes' pair of $\sigma ^0$.}
\end{ex}
 
 %\begin{rem}
{\rm It is clear that if a couple (SP, AP) is realizable, 
then it can be realized by a polynomial with all  simple roots,  
because the property of having non-vanishing coefficients is  preserved under small perturbations 
of the roots.}
%\end{rem} 

\medskip
In this short survey we present what is currently known about Problem~\ref{pb1}. After the pioneering observations of D.~J.~Grabiner \cite {Gr} which started this line of research, important contributions to Problem~\ref{pb1}  have been made by 
A.~Albouy and Y.~Fu \cite{AlFu} who, in particular, described all non-realizable combinations of the numbers of positive and negative roots and respective sign patterns up to degree $6$. Our  results on this topic which we summarize below  can be found in \cite{FoKoSh, FoKoSh1, ChGaKo} and \cite{KoDBAN07}--\cite{KoSh}. On the other hand, we find it surprising that such a natural classical question has not deserved any attention in the past and we hope that this survey will help to change the situation. 
 The current status of Problem~\ref{pb1} is not very satisfactory in spite of the complete results in degrees up to $8$ as well as several series of non-realizable cases in all degrees. There is still no general conjecture describing all non-realizable cases. It might happen that the answer to  Problem~\ref{pb1} in sufficiently high degrees is very complicated. 
 
On the other hand, besides Problem~\ref{pb1} as it is stated there is a significant number of related basic questions which can be posed in connection to the latter Problem and are still waiting for their researchers. (Very few  of them are listed in \S~\ref{sec:outook}.) 

\medskip
One should also add that there is a number of completely different directions in which mathematicians are trying to extend Descartes' rule of signs. They include, for example,  rule of signs for other univariate analytic functions including exponential functions, trigonometric functions and orthogonal polynomials, multivariate Descartes' rule of signs, tropical rule of signs, rule of signs in the complex domain etc., see e.g.  \cite{Di, FNoSh, Ha} and references therein. But we think that Problem~\ref{pb1} is the closest one to the original investigations by R.~Descartes himself.

\medskip
The structure of the paper is as follows. In \S~\ref{sec2} we provide the  information about the solution of Problem~\ref{pb1} in  degrees up to 11. In  \S~\ref{sec3} we present several infinite series of non-realizable couples  (SP, AP). Finally, in \S~\ref{sec4} we discuss two generalizations of Problem~\ref{pb1} and their partial solutions.  

\section{Solution of  the realization Problem~\ref{pb1} in small degrees}\label{sec2}
\subsection{Natural 
$\mathbb{Z}_2\times \mathbb{Z}_2$-action and 
degrees $d=1$, $2$ and $3$}

Let us start with the following useful observation. 

{\rm To shorten the list of cases (SP, AP) under consideration, 
we can use the following 
$\mathbb{Z}_2\times \mathbb{Z}_2$-action whose first generator acts by 

\begin{equation}\label{eqfirstgenerator}
P(x)\, \, \mapsto \, \, (-1)^dP(-x)
\end{equation} 
and the second one acts by 

\begin{equation}\label{eqsecondgenerator}
P(x)\, \, \mapsto \, \, P^R(x)\, \, :=\, \, x^dP(1/x)/P(0)\, \, .
\end{equation} 
Obviously, the first generator  
exchanges the components of the AP.
Concerning the second generator, to obtain the SP defined by the polynomial 
$P^R$ one  
 has to read  the SP defined by $P(x)$ backwards. The roots of $P^R$ are the 
reciprocals of these of $P$ which implies that both polynomials have  the same 
numbers of positive and negative roots. Therefore   the SPs which they define 
have  the same AP. 

\begin{rem}{\rm 
A priori the length of an orbit of any $\mathbb{Z}_2\times \mathbb{Z}_2$-action 
could be  $1$, $2$ or $4$, 
but for the above action, orbits of length $1$ do not exist since the second components of the SPs 
defined by the polynomials $P(x)$ and $(-1)^dP(-x)$ are always different. 
When an orbit of length $2$ occurs and $d$ is even, then both SPs are 
symmetric w.r.t. their 
middle points (hence their last components equal $+$).  Similarly when $d$ is odd, then 
one of the two SPs is symmetric w.r.t. 
its middle (with the last component equal to $+$) and the other one is 
anti-symmetric. Thus,  its last component equals~$-$.}Ê
\end{rem}

It is obvious that all pairs or quadruples  (SP, AP) constituting  a given orbit are 
simultaneously (non-)realizable.}

\medskip
As a warm-up exercise, let us consider degrees $d=1,2$ and $3$. In these cases, the answer to Problem~\ref{pb1} is positive. We 
give the list of SPs, with the respective values $c$ and $p$  
of their APs, and examples of polynomials realizing the couples (SP, AP). 
In order to shorten the list we consider only SPs beginning with two $+$ signs;  
the cases when these signs are $(+,-)$ are realized by the 
respective polynomials $(-1)^dP(-x)$. All quadratic factors in the Table below 
have no real roots.

$$\begin{array}{cccccc}
d&{\rm SP}&c&p&{\rm AP}&P\\ \\ 
1&(+,+)&0&1&(0,1)&x+1\\ \\  
2&(+,+,+)&0&2&(0,2)&x^2+3x+2=(x+1)(x+2)\\ \\ 
&&&&(0,0)&x^2+x+1\\ \\ 
&(+,+,-)&1&1&(1,1)&x^2+x-2=(x-1)(x+2)\\ \\ 
3&(+,+,+,+)&0&3&(0,3)&x^3+6x^2+11x+6=(x+1)(x+2)(x+3)\\ \\ 
&&&&(0,1)&x^3+3x^2+4x+2=(x+1)(x^2+2x+2)\\ \\ 
&(+,+,+,-)&1&2&(1,2)&x^3+2x^2+x-6=(x-1)(x+2)(x+3)\\ \\ 
&&&&(1,0)&x^3+5x^2+4x-10=(x-1)(x^2+6x+10)\\ \\ 
&(+,+,-,+)&2&1&(2,1)&x^3+x^2-24x+36=(x+6)(x-2)(x-3)\\ \\ 
&&&&(0,1)&x^3+2x^2-19x+30=(x+6)(x^2-4x+5)\end{array}$$ 
$$\begin{array}{cccccc}&(+,+,-,-)&1&2&(1,2)&x^3+x^2-4x-4=(x-2)(x+1)(x+2)\\ \\ 
&&&&(1,0)&x^3+2x^2-3x-10=(x-2)(x^2+4x+5)
\end{array}$$

\begin{ex}
{\rm For $d=4$, an example 
of an orbit of length $2$ is given by the couples 
$$((+,-,-,-,+),(2,2))~~~\, {\rm and}~~~\, ((+,+,-,+,+),(2,2)).$$
Here both SPs are symmetric w.r.t. its middle.

\medskip
For $d=5$, such an example is given by the couples 
$$((+,-,-,-,-,+),(2,3))~~~\, {\rm and}~~~\, ((+,+,-,+,-,-),(3,2)).$$
The first of the SPs is symmetric and the second one is anti-symmetric 
w.r.t. their middles. 
  
  \medskip
Finally, for $d=3$, the following four couples  (SP, AP): 
$$\begin{array}{lclc}
((+,+,+,-), (1,2));& 
((+,-,+,+), (2,1));\\ \\ 
((+,-,-,-), (1,2));&((+,+,-,+), (2,1)).\end{array}$$
constitute one orbit for $d=3$. In this example all admissible pairs are 
Descartes' pairs.}  
\end{ex}

\medskip
\subsection{Degrees $d\geq 4$}

It turns out that for $d\geq 4$, it is no longer true that all couples 
(SP, AP) are 
realizable by  polynomials of degree $d$. Namely, the following result 
can be found in~\cite{Gr}:

\begin{tm}\label{tmd=4}
The only  couples (SP, AP) which are non-realizable by 
univariate polynomials of degree $4$ are:

$$((+,-,-,-,+), (0,2))~~~\, {\rm and}~~~\, ((+,+,-,+,+), (2,0)).$$
\end{tm}
It is clear that these two cases constitute one orbit of the 
$\mathbb{Z}_2\times \mathbb{Z}_2$-action of length $2$ (the SPs are the same 
when read the usual way and backwards). 

\begin{proof}
The argument showing 
non-realizability in Theorem~\ref{tmd=4} is easy. Namely,   
if a polynomial 

$$P\, \, :=\, \, x^4~+~a_3x^3~+~a_2x^2~+~a_1x~+~a_0$$ 
realizes the second of these couples and has two 
positive roots $\alpha <\beta$ and no negative roots, then for any 
$u\in (\alpha ,\beta )$, the values of the monomials $x^4$, $a_2x^2$ and $a_0$ 
are the same at $u$ and $-u$ while the monomials $a_3x^3$ and $a_1x$ 
are positive at $u$ and negative at $-u$. Hence $P(-u)<P(u)<0$. As $P(0)>0$ 
and $\lim _{x\rightarrow -\infty}P(x)=+\infty$, the polynomial $P$ has two 
negative roots as well -- a contradiction. 

For $d=4$, realizability of all other couples (SP, AP) can be proved by 
producing explicit examples.\end{proof} 

\begin{rem}\label{remdiscr}
{\rm In \cite{KoSh} a geometric illustration of 
the non-realizability of the two cases mentioned in Theorem~\ref{tmd=4} is 
proposed. Namely, one considers the family of polynomials 
$Q:=x^4+x^3+ax^2+bx+c$ 
and the {\em discriminant set} 

$$\Delta \, \, :=\, \, \{ \, (a,b,c)\in \mathbb{R}^3\, |\, 
{\rm Res}(Q,Q')\, =\, 0\, \} \, \, ,$$
where Res\, $(Q,Q')$ is the resultant of the polynomials $Q$ and $Q'$. The 
hypersurface $\Delta =0$ partitions $\mathbb{R}^3$ into three open domains, 
in which the polynomial $Q$ has $0$, $1$ or $2$ complex conjugate 
pairs of roots respectively. These domains intersect the $8$ open orthants of 
$\mathbb{R}^3$ defined by the coordinate system $(a, b, c)$, 
and in each of these 
intersections the polynomial $Q$ has one and the same number of positive, 
negative and complex roots, as well as the same signs of its coefficients. 
The non-realizability of the couple 
$((+,+,-,+,+), (2,0))$ can be interpreted as the fact that the corresponding 
intersection is empty. Pictures of discriminant sets allow to construct 
easily the numerical examples mentioned in the proof of Theorem~\ref{tmd=4}.

It remains to notice that for $\alpha >0$, $\beta >0$, 
the polynomuials $P(x)$ and 
$\beta P(\alpha x)$ have one and the same numbers of positive, negative 
and complex roots. Therefore it suffices to consider the family of polynomials 
$Q$ in order to cover all SPs beginning with $(+,+)$. The ones beginning with 
$(+,-)$ will be covered by the family $Q(-x)$.}
\end{rem}

\medskip
For degrees $d=5$ and $6$, the following result can be found in \cite{AlFu}. 

\begin{tm}\label{tmd=56}
{\rm (1)} The only two  couples (SP, AP) 
which are non-realizable 
by univariate polynomials of degree $5$ are:

$$((+,-,-,-,-,+), (0,3))~~~\, {\rm and}~~~\, ((+,+,-,+,-,-), (3,0)).$$

\noindent
{\rm (2)} For degree $d=6$, up to the above $\mathbb{Z}_2\times \mathbb{Z}_2$-action, 
the only non-realizable  couples (SP, AP) are:

$$\begin{array}{lclc}
((+,-,-,-,-,-,+), (0,2));&((+,-,-,-,-,-,+), (0,4));\\ \\ 
((+,-,+,-,-,-,+), (0,2));&((+,+,-,-,-,-,+), (0,4)).
\end{array}$$ 
\end{tm}      
The two cases of Part (1) of Theorem~\ref{tmd=56} also form an orbit of the 
$\mathbb{Z}_2\times \mathbb{Z}_2$-action of length $2$. Each of the first two 
cases of Part (2) defines an orbit of length $2$ while each of the last two 
cases defines an orbit of length~$4$.

\medskip
For $d=7$, the following theorem is contained in \cite{FoKoSh}:

\begin{tm}
For univariate polynomials of degree $7$, among their 1472 possible couples (SP, AP) (up to the 
$\mathbb{Z}_2\times \mathbb{Z}_2$-action) exactly the following $6$ 
are non-realizable: 

$$\begin{array}{lclc}
((+,+,-,-,-,-,-,+), (0,5));&((+,+,-,-,-,-,+,+), (0,5));\\ \\ 
((+,-,-,-,-,+,-,+), (0,3));&((+,+,+,-,-,-,-,+), (0,5));\\ \\ 
((+,-,-,-,-,-,-,+), (0,3));&((+,-,-,-,-,-,-,+), (0,5)).\end{array}$$
\end{tm}

The lengths of the respective orbits in these $6$ cases are 
$4$, $2$, $4$, $4$, $2$ and~$2$. 

\medskip
The case $d=8$ has been partially solved in \cite{FoKoSh} and completely 
in~\cite{KoCMJ}:

\begin{tm}
For degree $d=8$, among the 3648 possible couples (SP, AP) (up to the 
$\mathbb{Z}_2\times \mathbb{Z}_2$-action) exactly the following $19$ 
are non-realizable:

$$\begin{array}{lclc}
((+,+,-,-,-,-,-,+,+), (0,6));&((+,+,-,-,-,-,-,-,+), (0,6));\\ \\ 
((+,+,+,-,-,-,-,-,+), (0,6));&((+,+,+,+,-,-,-,-,+), (0,6));\\ \\ 
((+,-,+,-,-,-,+,-,+), (0,2));&((+,-,+,-,+,-,-,-,+), (0,2));\\ \\ 
((+,-,+,-,-,-,-,-,+), (0,2));&((+,-,+,-,-,-,-,-,+), (0,4));\\ \\ 
((+,-,-,-,+,-,-,-,+), (0,2);&((+,-,-,-,+,-,-,-,+), (0,4));\\ \\ 
((+,-,-,-,-,-,-,-,+), (0,2));&((+,-,-,-,-,-,-,-,+), (0,4);\\ \\ 
((+,-,-,-,-,-,-,-,+), (0,6));&((+,+,+,-,-,-,-,+,+), (0,6));\\ \\ 
((+,-,-,-,-,+,-,-,+), (0,4));&((+,-,-,-,-,-,-,+,+), (0,4));\\ \\ 
((+,-,+,+,-,-,-,-,+), (0,4));&((+,-,+,-,-,-,-,+,+), (0,4));\\ \\ 
((+,-,-,-,-,+,-,+,+), (0,4)).&&
\end{array}$$
\end{tm}

The lengths of the respective orbits are: $2$, $4$, $4$, $4$, $2$, $4$, $4$, 
$4$, $2$, $2$, $2$, $2$, $2$, $4$, $4$, $4$, $4$, $4$ and~$4$.

\begin{rem}
{\rm As we see above,  for 
$d=4$, $5$, $6$, $7$ and $8$,  up to the 
$\mathbb{Z}_2\times \mathbb{Z}_2$-action, the numbers of non-realizable 
cases are 
 $1$, $1$, $4$, $6$ and $19$ respectively. The fact that these numbers increase 
more when $d=5$ and $d=7$ than when $d=4$ and $d=6$ could be related to the 
fact that the maximal possible number of complex conjugate pairs of roots of 
a real univariate degree $d$ polynomial is $[d/2]$.% (the integer part of $d/2$). 
  This number increases w.r.t. $[(d-1)/2]$ 
 when $d$ is even and does not increase  when $d$ is odd.}
\end{rem}  

Observe that for $d\leq 8$,  all examples of couples (SP, AP) which 
are non-realizable are 
with APs of the form $(\nu ,0)$ or $(0,\nu )$, $\nu \in \mathbb{N}$. 
Initially we thought that this is always the case. However recently 
it was proven that, for higher degrees, this fact is no longer true, 
see~\cite{KoMB}:

\begin{tm}
For $d=11$, the following couple (SP, AP) 

$$((+,-,-,-,-,-,+,+,+,+,+,-), (1,8))$$
is non-realizable.  The Descartes' pair in this case equals $(3,8)$. 
\end{tm}

\medskip
There is a strong evidence 
that for $d=9$, the similar couple (SP, AP)    

$$((+,-,-,-,-,+,+,+,+,-)~~~\, ,~~~\, (1,6))$$
is also non-realizable. (Its Descartes' pair equals $(3,6)$.)  
If this were true, then $9$ would be the 
smallest degree with an example of a non-realizable couple (SP, AP) 
for which both components of the AP are nonzero. When studying the cases 
$d=8$ and $d=11$ (see~\cite{KoCMJ} and \cite{KoMB}) discriminant sets have 
been considered, see Remark~\ref{remdiscr}.

\medskip
Summarizing the above, we have to admit that  the  information 
in low degrees available at the moment does not allow us 
to formulate a consistent conjecture describing all non-realizable 
couples in  an arbitrary degree which we could consider as sufficiently 
well-motivated. 

\section{Series of examples of (non-)realizable 
couples (SP, AP)}\label{sec3}

In this section we present  series of couples (non-)realizable for infinitely 
many  degrees.  We decided to include those proofs of the statements formulated below which are short 
and instructive. 

\subsection{Some examples of realizability and a concatenation lemma}

Our first examples of realizability deal with polynomials with the minimal 
possible number of real roots:

\begin{prop}
For $d$ even, any SP whose last component is a $+$ (resp. is a $-$) 
is realizable with the 
AP $(0,0)$ (resp. $(1,1)$). For $d$ odd, any SP whose last component 
is a $+$ (resp. is a $-$) 
is realizable with the AP $(0,1)$ (resp. $(1,0)$).
\end{prop}

\begin{proof}
Indeed, for any given SP,  it suffices to choose any polynomial defining 
this SP 
and to  increase (resp. decrease)  its constant term sufficiently much 
if the latter is positive (resp. negative). The resulting polynomial 
will have the required number of real roots. 
\end{proof}

Our next example deals with {\em hyperbolic} polynomials, i.e., 
real polynomials with all real roots. Several topics concerning hyperbolic 
polynomials are developed in~\cite{KoPS}.

\begin{prop}\label{prophyp}
Any SP is realizable with its Descartes' pair.
\end{prop}

 Proposition~\ref{prophyp} will follow from   the following 
{\em concatenation lemma} 
whose proof can be found in~\cite{FoKoSh}.

\begin{lm}\label{lm:conc}
Suppose that  monic polynomials $P_1$ and $P_2$, of degrees $d_1$ and $d_2$ 
resp., 
realize the SPs $(+,\hat{\sigma}_1)$ and $(+,\hat{\sigma}_2)$, where 
$\hat{\sigma}_j$ are the SPs defined by $P_j$ in which the first $+$ 
is deleted. Then

\smallskip
\noindent
{\rm (1)} if the last position of $\hat{\sigma}_1$ is a $+$, then for any 
$\varepsilon >0$ small enough, the polynomial 
$\varepsilon ^{d_2}P_1(x)P_2(x/\varepsilon )$ realizes 
the SP $(+,\hat{\sigma}_1,\hat{\sigma}_2)$ and 
the AP $(pos_1+pos_2,neg_1+neg_2)$;
 
 \smallskip
 \noindent
{\rm (2)} if the last position of $\hat{\sigma}_1$ is a $-$, then for any 
$\varepsilon >0$ small enough, the polynomial 
$\varepsilon ^{d_2}P_1(x)P_2(x/\varepsilon )$ realizes 
the SP $(+,\hat{\sigma}_1,-\hat{\sigma}_2)$ and 
the AP $(pos_1+pos_2,neg_1+neg_2).$ (Here $-\hat{\sigma}_2$ is the SP 
obtained from $\hat{\sigma}_2$ by changing each $+$ by a $-$ and vice versa.)
\end{lm}

The concatenation lemma allows to deduce the realizability of couples (SP, AP) 
with higher 
values of $d$ from that of couples with smaller $d$ in which cases explicit 
constructions are usually easier to obtain. On the other hand, non-realizability of special cases cannot be 
concluded using this lemma.  

\begin{ex}\label{ex:imp} {\rm  Denote by $\tau$ the last entry of the SP 
$ \tsi_1$. 
We consider the cases 

$$\begin{array}{cccccccccc}
P_2(x)&=&x-1,&x+1,&x^2+2x+2,&x^2-2x+2&{\rm with}\\ \\  
(pos_2,neg_2)&=&(1,0),&(0,1),&(0,0) ,&(0,0)&{\rm resp.}\end{array}$$ 
When $\tau=+,$ then one has 
respectively 

$$\tsi_2\, \, =\, \, (-)\, \, ,\, \, (+)\, \, ,\, \, (+,+)\, \, ,\, \, (-,+)$$ 
and the SP of  
$\eps^{d_2}P_1(x)P_2(x/\eps)$ equals 

$$(+, \tsi_1,-)\, \, ,\, \, (+, \tsi_1,+)\, \, ,\, \,  
(+, \tsi_1,+,+)\, \, ,\, \, (+, \tsi_1,-,+)\, \, .$$ 
When $\tau=-,$ then one has respectively  

$$ \tsi_2\, \, =\, \, (+)\, \, ,\, \, (-)\, \, ,\, \, (-,-)\, \, ,\, \, (+,-)$$ 
and the SP of  
$\eps^{d_2}P_1(x)P_2(x/\eps)$ equals 

$$(+, \tsi_1,+)\, \, ,\, \, (+, \tsi_1,-)\, \, ,\, \,  
(+, \tsi_1,-,-)\, \, ,\, \, (+, \tsi_1,+,-)\, \, .$$  }  
\end{ex}

\begin{proof}[Proof of Proposition~\ref{prophyp}]
We will use induction on the degree $d$ of the polynomial. 
For $d=1$, the SP $(+,-)$ (resp. $(+,+)$) is realizable with the AP $(1,0)$ 
(resp. $(0,1)$) by the polynomial $x-1$ 
(resp. $x+1$). 

For $d=2$, we apply 
Lemma~\ref{lm:conc}. Set $P_1:=x+1$ and $P_2:=x-1$. Then for 
$\varepsilon >0$ small enough, the polynomials 

$$\begin{array}{cccccc}
\varepsilon P_1(x)P_2(x/\varepsilon )&=&(x+1)(x-\varepsilon )&=&
x^2+(1-\varepsilon )x-\varepsilon&{\rm and}\\ \\ 
\varepsilon P_2(x)P_1(x/\varepsilon )&=&(x-1)(x+\varepsilon )&=&
x^2+(-1+\varepsilon )x-\varepsilon& \end{array}$$
define the SPs $(+,+,-)$ and $(+,-,-)$ respectively 
and realize them with the AP $(1,1)$. In the same way one can concatenate 
$P_1$ (resp. $P_2$) with itself to realize the SP $(+,+,+)$ with the AP 
$(0,2)$ (resp. the SP $(+,-,+)$ with the AP $(2,0)$). These are all possible 
cases of monic hyperbolic degree $2$ polynomials 
with nonvanishing coefficients.   

For $d\geq 2$,  in order  
to realize a SP $\sigma$ with its Descartes' pair $(c,p)$, we represent 
$\sigma$ in the form $(\sigma ^{\dagger},u,v)$, where $u$ and $v$ are the last 
two components of $\sigma$ and $\sigma ^{\dagger}$ is the SP obtained from 
$\sigma$ by deleting $u$ and $v$. Then we choose $P_1$ to be a monic 
polynomial realizing the SP $(\sigma ^{\dagger},u)$ 

(i) with the AP $(c-1,p)$, and 
we set $P_2:=x-1$, if $u=-v$; 

(ii) with the AP $(c,p-1)$, and we set $P_2:=x+1$, if $u=v$.
\end{proof}

Our next result discusses (non-)realizability for polynomials with only 
two sign changes (see \cite{FoKoSh} and~\cite{FoKoSh1}). 

\begin{prop}\label{prop:3parts} Consider a sign pattern $\bsi$ with $2$ 
sign changes, consisting of $m$ consecutive pluses  
followed by $n$ consecutive minuses and then by $q$ consecutive pluses, 
where $m+n+q=d+1.$ Then 

\noindent 
{\rm (i)} for the pair $(0,d-2),$ this sign pattern is not realizable if 
\begin{equation}\label{eq:kappa}
\kappa:=\frac{d-m-1}{m}\cdot \frac{d-q-1}{q} \ge 4;
\end{equation}

\noindent 
{\rm (ii)} the sign pattern $\bsi$ is realizable with any pair of the form 
$(2,v)$, except in the case when $d$ and $m$ are even, $n=1$ 
(hence $q$ is even) and $v=0$. 
\end{prop} 

Certain results about realizability are formulated in terms of the ratios 
between the quantities $pos$, $neg$ and $d$. 
The following proposition is proved in~\cite{FoKoSh}:

\begin{prop}\label{propratio}
For a given couple (SP, AP), if $\min (pos, neg)>[(d-4)/3]$, then this 
couple is realizable.
\end{prop}

\subsection{The even and the odd series}

Suppose that the degree $d$ is even. Then the following result holds, 
(see Proposition~4 in~\cite{FoKoSh}):

\begin{prop}\label{prop:even}
Consider the SPs satisfying the following three conditions: 

\noindent
{\rm (i)} their last entry (i.e. the sign of the constant term) is a $+$;

\noindent
{\rm
(ii)} the signs of all odd monomials are $+$;

\noindent
{\rm
(iii)} among the remaining signs of even monomials there are exactly 
$\ell \geq 1$ signs $-$ (at arbitrary positions).

Then, for any such SP, the APs $(2,0), (4,0), \ldots, (2\ell ,0)$, and 
only they, are non-realizable.
\end{prop}

Suppose now that the degree $d\geq 5$ is odd. For 
$1\leq k\leq (d-3)/2$, denote
by $\sigma _k$ the SP beginning with two pluses followed by $k$ pairs "$(-,+)$" 
and then by $d-2k-1$ minuses. Its Descartes' pair of $\sigma _k$ 
equals $(2k+1, d-2k-1)$. 
The following proposition is 
proved in~\cite{KoSh}:

\begin{tm} \label{th:series}

\noindent
{\rm (1)} The SP $\sigma _k$ is not realizable with any of the pairs 
$(3, 0), (5, 0), \ldots,$ $ (2k+1, 0)$;

\noindent
{\rm
(2)} The SP $\sigma _k$ is realizable with the pair $(1, 0)$;

\noindent
{\rm
(3)} The SP $\sigma _k$ is realizable with any of the APs 
$(2\ell + 1, 2r)$, $\ell = 0, 1, \ldots, k$, $r=1, 2, \ldots, (d-2k-1)/2$.
\end{tm}

One can observe that Cases (1), (2) and (3) exhaust all possible 
APs $(pos, neg)$.

\section{Similar realization problems}\label{sec4}

In this section we consider realization problems similar or motivated by Problem~\ref{pb1}. A priori  it is hard to tell which of these or similar problems might have a reasonable answer. 

\subsection{$\mathcal{D}$-sequences} 

Consider a real  polynomial $P$ of degree $d$ and its derivative. By Rolle's 
theorem, if $P$ has exactly $r$ real roots 
(counted with multiplicity), then the derivative $P'$ 
has $r-1+2\ell$ real roots (counted with multiplicity), where 
$\ell \in \mathbb{N}\cup 0$. It is possible that  $P'$  has more 
real roots than $P$. For example,  for $d=2$ and $P=x^2+1$,  one gets $P^\prime=2x$ which has 
a real root at $0$ while $P$ has no real roots at all. For $d=3$, the polynomial 
$P=x^3+3x^2-8x+10=(x+5)((x-1)^2+1)$  
has one negative root and one 
complex conjugate pair, while its derivative 
$P'=3x^2+6x-8$ has one positive and one negative root. 

\medskip
Now, for $j=0$, $\ldots$, $d-1$, denote by $r_j$ and $c_j$ the numbers of real roots 
and  complex conjugate pairs of roots of the polynomial $P^{(j)}$ 
(both counted with multiplicity). These 
numbers satisfy the conditions 

\begin{equation}\label{Dseq}
r_j\leq r_{j+1}+1~~~\, ,~~~\, r_j+2c_j=d-j~.
\end{equation}

\begin{defi}
{\rm A sequence 
$((r_0,2c_0),(r_1,2c_1), \ldots,$ $ (r_{d-1},2c_{d-1}))$ satisfying 
conditions~(\ref{Dseq}) will be called  a {\em $\mathcal{D}$-sequence of length $d$}. We say that a given $\mathcal{D}$-sequence of length 
$d$ is {\em realizable} if there exists a real 
polynomial $P$ of degree $d$  with this $\mathcal{D}$-sequence, where for 
$j=0, \ldots, d-1$, all roots of $P^{(j)}$ are distinct.}
\end{defi}

\begin{ex}
{\rm One has $r_{d-1}=1$ and $c_{d-1}=0$. Clearly one has either $r_{d-2}=2$, 
$c_{d-2}=0$ 
or $r_{d-2}=0$, $c_{d-2}=1$. For small values of $d$, one has the following 
$\mathcal{D}$-sequences and respective polynomials realizing them:}

$$\begin{array}{lll}
d=1&(1,0)&x\\ \\ d=2&((2,0),(1,0))&x^2-1\\ \\ &((0,2),(1,0))&x^2+1\\ \\ 
d=3&((3,0),(2,0),(1,0))&x^3-x\\ \\ &((1,2),(0,2),(1,0))&x^3+x\\ \\  
&((1,2),(2,0),(1,0))&x^3+10x^2+26x.
\end{array}$$ 
\end{ex}  

\medskip
The following question a positive answer to which can be found   
in~\cite{KoDBAN07} seems very natural. 

\begin{problem}\label{pb2}
Is it true that for any $d\in \mathbb{N}$, any $\mathcal{D}$-sequence 
is realizable?
\end{problem}

\subsection{Sequences of admissible pairs}

Now we are going to formulate a problem which is a refinement of both 
Problems~\ref{pb1} and~\ref{pb2}. 

\smallskip
Recall that for a  real 
polynomial $P$ of degree $d$, the signs of its coefficients $a_j$ 
define the sign patterns $\sigma _0, 
\sigma _1, \ldots, \sigma _{d-1}$ corresponding to 
$P$ and to all its derivatives of order $\leq d-1$ since the SP 
$\sigma _j$ is obtained from $\sigma _{j-1}$ by deleting the last component. 
We denote by 
$(c_k, p_k)$ and $(pos_k, neg_k)$ the Descartes' and admissible pairs for 
the SPs $\sigma _k, \; k=0, \ldots, d-1$. The following restrictions 
follow from Rolle's theorem:  

\begin{equation}\label{Rolle1}
\begin{array}{lcl}
pos_{k+1}\geq pos_k-1&,&neg_{k+1}\geq neg_k-1\\ \\ {\rm and}&
pos_{k+1}+neg_{k+1}\geq pos_k+neg_k-1~.\end{array}
\end{equation}
It is always true that 

\begin{equation}\label{Rolle2}
pos_{k+1}+neg_{k+1}+3-pos_k-neg_k\in 2\mathbb{N}~.
\end{equation}   

\begin{defi}
{\rm Given a sign pattern $\sigma _0$ of length $d+1$, suppose that for 
$k=0, \ldots, d-1$, the pair 
$(pos_k,neg_k)$ satisfies the conditions 

\begin{equation}\label{eqDescartes}
\begin{array}{cccc}
pos_k\leq c_k,&c_k-pos_k\in 2\mathbb{Z},\\ \\ 
neg_k\leq p_k,&p_k-neg_k\in 2\mathbb{Z},\\ \\ 
{\rm and}&{\rm sgn}~a_k=(-1)^{pos_k}.\end{array}
\end{equation}
as wel as the inequalities (\ref{Rolle1}) -- (\ref{Rolle2}). 
Then we say that 
\begin{equation}\label{eq*}
(~(pos_0,neg_0)~~~,~~~\ldots ~~~,~~~(pos_{d-1},neg_{d-1})~)
\end{equation} 
is a {\em sequence of admissible pairs (SAP)}. In other words, it is  
 a sequence of pairs admissible for the sign pattern $\sigma _0$ 
in the sense of these conditions. We say 
that a given couple (SP, SAP) is {\em realizable} 
if there exists a polynomial $P$ whose coefficients 
have signs given by the SP $\sigma _0$ and such that 
for $k=0, \ldots, d-1$, the polynomial $P^{(k)}$ has exactly $pos_k$ 
positive and $neg_k$ negative roots, all of them being simple. Complex roots 
are also supposed to be distinct.}
\end{defi}

\begin{rem}\label{remformula}
{\rm If one only knows  the SAP $(\ref{eq*})$, 
the SP $\sigma _0$ can be restituted 
by the formula}  
$$\sigma _0=(~+~,~(-1)^{pos_{d-1}}~,~(-1)^{pos_{d-2}}~,~\ldots ~,~(-1)^{pos_0}~).$$ 

{\rm Nevertheless in order to make comparisons with 
Problem~\ref{pb1} more easily, we 
consider couples (SP, SAP) instead of just SAPs.   
But for a given SP,  there are, in general, 
several possible SAPs which is illustrated by the following example.}
\end{rem}

\begin{ex}\label{exSAPAP}
{\rm Consider the SP of length $d+1$ with all pluses. For $d=2$ and $3$, 
there are respectively two and three possible SAPs: 

$$\begin{array}{llllll}
((0,2),(0,1))&,&((0,0),(0,1))&,&&{\rm for~}d=2\\ \\ {\rm and}&&&&& \\ \\  
((0,3),(0,2),(0,1))&,&  
((0,1),(0,2),(0,1))&,&((0,1),(0,0),(0,1))&{\rm for}~d=3~.\end{array}$$ 

 For $d=4,5,6,7,8,9,10$, the  
numbers $A(d)$ of SAPs compatible with 
the  SP of length $d+1$ having all pluses are  

$$7,~~~\, ~~~\, 12,~~~\, ~~~\, 30,~~~\, ~~~\, 55,~~~\, ~~~\,  
143,~~~\, ~~~\, 273,~~~\, {\rm and}~~~\, 728$$
respectively. 
One can show  that $A(d)\geq 2A(d-1)$, if $d\geq 2$ is even, and 
$A(d)\geq 3A(d-1)/2$, if $d\geq 3$ is odd, see~\cite{ChGaKo}.}
\end{ex}     

\begin{ex}\label{exSAPAPbis}
{\rm There are two couples (SP, SAP) corresponding to the couple (SP, AP) 
$C:=((+,+,-,+,+)$, $(0,2))$; we also say  that the couple $C$ 
{\em can be extended} into these couples (SP, SAP). These are} 

$$\begin{array}{cccccccccccc}
(&(+,+,-,+,+)&,&(0,2)&,&(2,1)&,&(1,1)&,&(0,1)&)&{\rm and}\\ \\ 
(&(+,+,-,+,+)&,&(0,2)&,&(0,1)&,&(1,1)&,&(0,1)&)&.\end{array}$$

{\rm Indeed, by Rolle's theorem, the derivative of a polynomial realizing 
the couple $C$ has at least one negative root. 
By conditions (\ref{eqDescartes}) 
this derivative (whose degree equals 3) has an even number of 
positive roots. This yields just two possibilities for $(pos _1, neg_1)$, namely $(2,1)$ and $(0,1)$. 
The second derivative is a quadratic polynomial 
with positive leading coefficient and negative 
constant term. Hence it has a positive and a negative 
root. The realizability of the above two couples (SP, SAP) 
is proved in~\cite{ChGaKo}.}
\end{ex}  

Our final realization problem is as follows:

\begin{problem}\label{pb3}
For a given degree $d$, which couples (SP, SAP) are realizable?
\end{problem}

\begin{rems}\label{remrefinement} 
{\rm (1) This problem is a refinement of Problem~\ref{pb1}, 
because one considers 
the APs of the derivatives of all orders, and not just the one of 
the polynomial itself, see Remark~\ref{remformula}. 
Therefore if a given couple (SP, AP) is non-realizable, then all couples 
(SP, SAP) corresponding to it in the sense of Example~\ref{exSAPAPbis} are 
automatically  non-realizable. \\  

\noindent
(2) Obviously Problem~\ref{pb3} is a refinement of Problem~\ref{pb2} -- in the 
latter case one does not take into account the signs of the real roots of the 
polynomial and its derivatives.\\ 

\noindent 
(3) When we deal with couples (SP, SAP),  we can use the $\mathbb{Z}_2$-action 
defined by (\ref{eqfirstgenerator}). Therefore  it suffices to consider the cases 
of SPs beginning with $(+,+)$. The generator (\ref{eqsecondgenerator}) of the 
$\mathbb{Z}_2\times \mathbb{Z}_2$-action cannot be used, because when the 
derivatives of a polynomial are involved, the polynomial loses its last 
coefficients. Due to this circumstance the two ends of the SP  cannot be treated equally.}
\end{rems}

The following proposition is proved in~\cite{ChGaKo}:

\begin{prop}\label{mainprop1}
For any given SP of length $d+1$, $d\geq 1$, there exists a unique SAP such 
that $pos_0+neg_0=d$. This SAP is realizable. For the given SP, 
this pair $(pos_0, neg_0)$ is its Descartes' pair. 
\end{prop}

\begin{ex}
{\rm For even $d$,  consider  the  SP with all pluses. Any hyperbolic polynomial 
with all negative and distinct  roots realizes this SP with SAP 

$$(\, (0,d)\, ,\, (0,d-1)\, ,\, \ldots \, ,\, (0,1)\, )\, .$$
One can choose such a polynomial $P$ with all $d-1$ distinct critical values. 
Hence in the family of polynomials $P+t$, $t>0$, one encounters polynomials 
realizing this SP with any of the SAPs 
$$(\, (0,d-2\ell )\, ,\, (0,d-1)\, ,\, (0,d-2)\, ,\, 
\ldots \, ,\, (0,1)\, )\, \, ,\, \, 
\ell \, =\, 0\, ,\, 1\, ,\, \ldots \, d/2\, .$$
In the same way, for odd $d$,  
the SP $(+,+,\ldots ,+,-)$ is realizable with 
the SAP 

$$(\, (1,d-1)\, ,\, (0,d-1)\, ,\, (0,d-2)\, ,\, \ldots \, ,\, (0,1)\, )$$
by some hyperbolic polynomial $R$ 
with all  distinct roots and  critical values. 
In the family of polynomials $R-s$, $s>0$,  one encounters polynomials 
realizing this SP with any of the SAPs} 
$$(\, (1,d-1-2\ell )\, ,\, (0,d-1)\, ,\, (0,d-2)\, ,\, 
\ldots \, ,\, (0,1)\, )\, \, ,\, \, 
\ell \, =\, 0\, ,\, 1\, ,\, \ldots \, (d-1)/2\, .$$
\end{ex}

For $d\leq 5$,  the following exhaustive answer to Problem~\ref{pb3}  
is given in \cite{ChGaKo}:\\  

(A) For $d=1$, $2$ and $3$, all couples (SP, SAP) are realizable. \\  

(B) For $d=4$, the couple (SP, SAP) 

%\begin{equation}\label{exd4}
$$(~(+,+,-,+,+)~,~(2,0)~,~(2,1)~,~(1,1)~,~(0,1)~)~,$$
and only it (up to the $\mathbb{Z}_2$-action), 
is non-realizable. Its non-realizability follows from 
the one of the couple (SP, AP) 
$C^{\dagger}:=((+,+,-,+,+),(2,0))$, 
see Theorem~\ref{tmd=4}. 

One can observe that the couple $C^{\dagger}$ 
can be uniquely extended  into a couple (SP, SAP). Indeed, the first 
derivative has a positive constant term hence an even number of positive roots. 
This number is positive by Rolle's theorem. Hence the AP of the first 
derivative is $(2,1)$. In the same way,  
one obtains the APs $(1,1)$ and $(0,1)$ 
for the second and third derivatives respectively. \\ 

(C) For $d=5$, the following couples (SP, SAP), and only they, 
are non-realizable:

%\begin{equation}\label{exd45}
$$\begin{array}{ccccccccccccc}
(&(+,+,-,+,+,+)&,&(2,1)&,&(2,0)&,&(2,1)&,&(1,1)&,&(0,1)&)~,\\ \\ 
(&(+,+,-,+,+,+)&,&(0,1)&,&(2,0)&,&(2,1)&,&(1,1)&,&(0,1)&)~,\\ \\ 
(&(+,+,-,+,+,-)&,&(3,0)&,&(2,0)&,&(2,1)&,&(1,1)&,&(0,1)&)~,\\ \\ 
(&(+,+,-,+,+,-)&,&(1,0)&,&(2,0)&,&(2,1)&,&(1,1)&,&(0,1)&)~,\\ \\ 
(&(+,+,-,+,-,-)&,&(3,0)&,&(3,1)&,&(2,1)&,&(1,1)&,&(0,1)&)~.
\end{array}$$
%\end{equation}

\medskip
The non-realizability of the first four of them follows from that of the 
couple $C^{\dagger}$. The last one is implied by  part (1) of 
Theorem~\ref{tmd=56}; it is true that the couple (SP, AP) 
$((+,+,-,+,-,-),(3,0))$ extends in a unique way into a couple (SP, SAP), 
and this is the fifth of the  five  such couples cited above. 

One of the methods used in the study of couples (SP, AP) or (SP, SAP) is 
the explicit construction of polynomials with multiple roots which define 
a given SP. Such constructions are not difficult to carry out because one has 
to use families of polynomials with fewer 
parameters. Once a polynomial with multiple roots is constructed, one has 
to justify the possibility to deform it continuously into a nearby 
polynomial with all  distinct roots. Multiple roots can give rise to 
complex conjugate pairs of roots. An example of such a construction is the 
following lemma from~\cite{ChGaKo}:

\begin{lm}\label{lm23}
Consider the polynomials $S:=(x+1)^3(x-a)^2$ and $T:=(x+a)^2(x-1)^3$, $a>0$. 
Their coefficients of $x^4$ are  
positive if and only if respectively $a<3/2$ and $a>3/2$. 
The coefficients of the 
polynomial $S$ define the SP

$$\begin{array}{lllc}
(+,+,+,+,-,+)&{\rm for}&a\in (~0~,~(3-\sqrt{6})/3~)&,\\ \\  
(+,+,+,-,-,+)&{\rm for}&a\in (~(3-\sqrt{6})/3~,~3-\sqrt{6}~)&,\\ \\ 
(+,+,-,-,-,+)&{\rm for}&a\in (~3-\sqrt{6}~,~2/3~)&{\rm and}\\ \\ 
(+,+,-,-,+,+)&{\rm for}&a\in (~2/3~,~3/2~)&.\end{array}$$
The coefficients of $T$ define the SP

$$\begin{array}{lllc}
(+,+,-,+,+,-)&{\rm for}&a\in (~3/2,~(3+\sqrt{6})/3~)&,\\ \\  
(+,+,-,-,+,-)&{\rm for}&a\in (~(3+\sqrt{6})/3~,~3+\sqrt{6}~)&{\rm and}\\ \\ 
(+,+,+,-,+,-)&{\rm for}&a>3+\sqrt{6}&.\end{array}$$
\end{lm}

\section{Outlook}\label{sec:outook}

\noindent
{\bf 1.} Our  first open question deals with  the limit of the ratio between 
the quantities $R(d)$ of all realizable and $A(d)$ of all possible cases 
of couples (SP, AP) as $d\rightarrow \infty$. In principle, one does 
not have to take into account the 
$\mathbb{Z}_2\times \mathbb{Z}_2$-action in order not to face the problem of 
the two different possible lengths of orbits ($2$ and $4$). 

A priori, for $d\geq 4$, one  has $R(d)/A(d)\in (0,1)$. It would be interesting to find out whether this ratio has a limit as $d\rightarrow \infty$, and if 'yes`, 
whether this limit is $0$, $1$ or belongs to $(0,1)$. In the latter 
case it would be interesting to find the exact value. 

A less ambitious open problem is to find an interval %(if such an interval 
%exists) 
$[\alpha ,\beta ]\subset (0,1)$ to which this ratio belongs for 
any $d\in \mathbb{N}$, $d\geq 4$, or at least for $d$ sufficiently large.\\

\noindent 
{\bf 2.} A related problem  would be to find  sufficient 
conditions for 
realizability based on the ratios between the quantities $pos$, $neg$ and $d$. 
On one hand, when the ratios $pos/d$ and $neg/d$ are both large enough, 
one has realizability, see Proposition~\ref{propratio}. On the other hand, 
in all examples of non-realizability known up to now one of the quantities 
$pos$ and $neg$ is either $0$ or is very small compared to the other one. Thus it 
would be interesting to understand  the role of these ratios for 
the (non)-realizability of the couples (SP, AP). \\    

\noindent
{\bf
3.} Our third open question is about the realizability of couples (SP, SAP). 
For $d\leq 5$, the non-realizability of all non-realizable 
couples (SP, SAP) results 
from the non-realizability of the corresponding couples (SP, AP). In principle, 
one could imagine a situation in which there exists a couple (SP, AP)  
extending into several couples (SP, SAP) some of which are realizable and the 
remaining are not.  Whether for $d\geq 6$, such couples 
(SP, AP) exist or not is unknown at present.\\

\noindent
{\bf
4.} Our final natural and important question deals with the topology of intersections of the set of real univariant polynomials with a given number of real roots with orthants in the coefficient space (which means fixing the signs of the coefficients). It is well-known that the set of monic univariate polynomials of a given degree and with a given number of real roots is contractible. When we cut this set with the union of coordinate hyperplanes (coordinates being the coefficients of polynomials), then it splits into a number of connected components.  In each such connected component the number of positive and negative roots are fixed. But, in principle, it can happen that different connected components correspond to the same pair (pos, neg). Could this 
really happen? Are all such connected components contractible, or  they can have some non-trivial topology? 

\medskip\noindent
{\it Acknowldegements.}  The authors want to thank departments of mathematics of  Universit\'e C\^ote d'Azur and Stockholms Universitet for the hospitality,  financial support, and nice working conditions during our several visits to each other.

\end{document}